\documentclass[11pt,reqno]{amsart}

\usepackage[T1]{fontenc}
\usepackage[utf8]{inputenc}
\usepackage{amsmath,amssymb,amsfonts,amsthm,mathtools}
\usepackage[margin=1.1in]{geometry}
\usepackage[colorlinks=true,linkcolor=blue,citecolor=blue,urlcolor=blue]{hyperref}

\numberwithin{equation}{section}

\newtheorem{theorem}{Theorem}[section]
\newtheorem{proposition}[theorem]{Proposition}
\newtheorem{lemma}[theorem]{Lemma}
\newtheorem{corollary}[theorem]{Corollary}
\theoremstyle{remark}
\newtheorem{remark}[theorem]{Remark}

\DeclareMathOperator{\tr}{tr}
\newcommand{\R}{\mathbb R}
\newcommand{\Id}{\mathrm{Id}}
\newcommand{\ii}{\mathrm{i}}
\newcommand{\e}{\mathrm{e}}
\newcommand{\dd}{\,\mathrm{d}}
\newcommand{\norm}[1]{\left\|#1\right\|}
\newcommand{\abs}[1]{\left|#1\right|}
\newcommand{\la}[1]{\langle #1\rangle}

\title[Relativistic momentum-only Schauder counterexamples]{A note on relativistic kinetic Schauder estimates}
\author{Weinan Wang}
\address{Department of Mathematics, University of Oklahoma, Norman, OK, USA, 73072}
\email{ww@ou.edu}
\thanks{W. Wang was partially supported by the Simons Foundation TSM grant (No.~0007730).}
\date{\today}

\begin{document}

\begin{abstract}
In this note, motivated by \cite{DW26}, we construct counterexamples to momentum-only Schauder estimates for linear kinetic equations with relativistic transport. The construction exploits a structural compatibility between the nonlinear momentum-to-velocity map and a single fixed pair of smooth diffusion and drift coefficients: after conjugation, the diffusion becomes the Laplacian and all induced first-order terms cancel. Thus the obstruction is realized for one fixed operator, uniformly elliptic on every bounded momentum set, rather than through a frequency-dependent family of coefficients. We further show that the failure persists with zero external forcing by constructing uniformly positive stationary solutions for which the hidden spatial oscillation is encoded in a bounded zeroth-order coefficient whose momentum H\"older seminorm tends to zero. Consequently, the full Lorentzian H\"older control used in \cite{HSTT} cannot in general be replaced by momentum-only H\"older control.

\end{abstract}

\maketitle

\section{Introduction}\label{sec:intro}

In this note we study the regularity of linear kinetic equations with relativistic transport,
\begin{equation}\label{eq:rel-linear}
   \partial_t f+\frac{p}{\la p}\cdot\nabla_x f
   =\tr\!\left(A(p)D_p^2f\right)+B(p)\cdot\nabla_pf+s,
   \qquad (t,x,p)\in\R\times\R^3\times\R^3,
\end{equation}
where $\la p=(1+|p|^2)^{1/2}$.  Equations of this form underlie the regularity theory of the relativistic Landau equation \cite{SG04,ST19}.  Henderson, Snelson, Tarfulea, and Taskovi\'c \cite{HSTT} recently established conditional regularity and decay estimates for the relativistic Landau equation and, as a key linear ingredient, proved  Schauder estimate for equations of the form \eqref{eq:rel-linear} \cite[Theorem~6.5]{HSTT}.  That estimate is formulated in Lorentzian kinetic H\"older spaces, is singled out in \cite{HSTT} as being of independent interest, and, as in the Galilean theory, assumes an intrinsic H\"older modulus in all nontrivial variables---in particular a spatial modulus. In the Galilean case the transport is $v\cdot\nabla_x$ and the quantitative theory is organized by the kinetic dilations $(t,x,v)\mapsto(r^2t,r^3x,rv)$, which pair H\"older regularity $C_v^\alpha$ in velocity with $C_x^{\alpha/3}$ in space.  Kinetic Schauder estimates in this scaling, with progressively weaker regularity in time, were obtained in \cite{IM21,HW24,BB24,LPP23,DY25}.  The endpoint question---whether the spatial modulus can be removed entirely, so that $D_v^2f$ is controlled by velocity H\"older norms of the coefficients and forcing alone---was raised in \cite[Conjecture~1.3]{HW24} and recently answered in the negative by Dong and the author in \cite{DW26}: the velocity-only kinetic Schauder estimate is false, already for the constant-coefficient Kolmogorov equation with bounded forcing independent of the velocity variable.

The purpose of this note is to settle the same endpoint for relativistic transport.  We exhibit one fixed smooth operator of the form \eqref{eq:rel-linear}, with the coefficients \eqref{eq:AB-def}, uniformly elliptic on every fixed momentum ball, together with an explicit family of smooth stationary solutions $f_N$, driven by forcings $s_N$ independent of the momentum variable, such that on the unit kinetic cylinder
\[
   \norm{f_N}_{L^\infty}\le CN^{-2/3},
   \qquad
   [s_N]_{C_p^\alpha}=0,
   \qquad
   [D_p^2f_N]_{C_p^\alpha}\ge cN^{\alpha/3}.
\]
Hence no local Schauder estimate can control the momentum H\"older seminorm of the momentum Hessian by momentum H\"older norms of the solution, the coefficients, and the forcing alone (Theorem~\ref{thm:rel-counterexample} and Corollary~\ref{cor:no-momentum-only-nonlinear-control}).  The failure persists with zero source: the hidden spatial oscillation can be placed in a bounded zeroth-order coefficient whose momentum H\"older seminorm tends to zero (Proposition~\ref{prop:rel-zero-source}).

The construction is exact rather than perturbative.  The key observation, used systematically in \cite{HSTT}, is that
\[
   \Phi(p):=\frac{p}{\la p}
\]
maps momentum space onto the unit velocity ball and conjugates relativistic transport to Galilean kinetic transport.  We choose the coefficients $A$ and $B$ so that the complete transformed equation is the constant-coefficient Kolmogorov equation (Lemma~\ref{lem:conjugation}), and then transport the one-frequency construction of \cite{DW26}: an explicit spatial Fourier mode whose momentum Hessian oscillates by an order-one amount across the kinetic length scale $N^{-1/3}$, while the forcing is independent of momentum. The example identifies the missing quantity precisely.  The forcing $s_N=\cos(Nx_1)$ is invisible to every momentum seminorm, while its spatial H\"older seminorm at the kinetic exponent has size $N^{\alpha/3}$, exactly the size of the Hessian growth.  The construction is therefore consistent with the Lorentzian Schauder estimate of \cite{HSTT} and shows that its spatial modulus cannot be omitted (Remarks~\ref{rem:consistency} and~\ref{rem:lorentz}).
\section{An exact conjugation to the Kolmogorov equation}\label{sec:conjugation}

For $p\in\R^3$, set
\[
   \Phi(p)=\frac{p}{\la p},
   \qquad
   \Psi(v)=\frac{v}{\sqrt{1-|v|^2}},\quad |v|<1.
\]
Thus $\Phi$ and $\Psi$ are inverse diffeomorphisms between $\R^3$ and $B_1$.
For $r>0$, write
\[
   \mathcal Q_r=(-r^2,0]\times B_{r^3}^x\times B_r^p.
\]
For a function $h=h(t,x,p)$ on a product cylinder $Q$, define its momentum H\"older seminorm by
\[
   [h]_{C_p^\alpha(Q)}
   :=\sup_{\substack{(t,x,p),(t,x,p')\in Q\\0<|p-p'|<1/2}}
   \frac{|h(t,x,p)-h(t,x,p')|}{|p-p'|^\alpha}.
\]
The spatial seminorm $[h]_{C_x^\beta(Q)}$ is defined analogously, with the supremum taken over pairs $(t,x,p),(t,x',p)\in Q$ with $0<|x-x'|<1/2$.  For functions of $p$ alone, such as the coefficients below, the same momentum seminorm is used on a momentum ball.  For matrix-valued functions we use the Frobenius norm. Let
\begin{equation}\label{eq:AB-def}
   A(p):=\la p^2(\Id+p\otimes p)^2,
   \qquad
   B(p):=\bigl(3\la p^4+2\la p^2\bigr)p.
\end{equation}
Both coefficients are smooth.  On every fixed momentum ball they are bounded, and $A$ is uniformly elliptic.  In particular, on $B_1$,
\[
   \Id\le A(p)\le 8\Id.
\]

\begin{lemma}[Exact conjugation]\label{lem:conjugation}
Let $f(t,x,p)=u(t,x,\Phi(p))$.  Then, under the identification $v=\Phi(p)$, the equation
\begin{equation}\label{eq:exact-conj}
   \partial_t f+\frac{p}{\la p}\cdot\nabla_xf
   =\tr(A(p)D_p^2f)+B(p)\cdot\nabla_pf+s(t,x)
\end{equation}
holds if and only if
\begin{equation}\label{eq:kolmo}
   \partial_tu+v\cdot\nabla_xu=\Delta_vu+s(t,x).
\end{equation}
\end{lemma}
Next, we give the proof of Lemma \ref{lem:conjugation}.
\begin{proof}
First, we write
\[
   q=1-|v|^2,
   \qquad
   R(v)=\Id-v\otimes v.
\]
The change-of-variables formula in \cite[Lemma~6.3]{HSTT} gives
\[
   \widetilde A(v)=qR(v)A(\Psi(v))R(v)
\]
and
\[
\begin{aligned}
   \widetilde B(v)
   &=\sqrt q\,R(v)B(\Psi(v)) \\
   &\quad+q\Bigl[\bigl(3v\cdot A(\Psi(v))v-\tr A(\Psi(v))\bigr)v
          -2A(\Psi(v))v\Bigr].
\end{aligned}
\]
Since
\[
   R(v)^{-1}=\Id+p\otimes p,
   \qquad p=\Psi(v),
   \qquad q=\la p^{-2},
\]
the definition of $A$ is equivalent to
\[
   A(\Psi(v))=q^{-1}R(v)^{-2}.
\]
Hence $\widetilde A(v)=\Id$.

The matrix $R(v)$ has eigenvalues $1,1,q$, with $v$ in the radial eigendirection.  Therefore
\[
   A(\Psi(v))v=q^{-3}v,
   \qquad
   v\cdot A(\Psi(v))v=(1-q)q^{-3},
   \qquad
   \tr A(\Psi(v))=2q^{-1}+q^{-3}.
\]
It follows that
\[
   \Bigl(3v\cdot A(\Psi(v))v-\tr A(\Psi(v))\Bigr)v
   -2A(\Psi(v))v
   =-\frac{2q+3}{q^2}v.
\]
On the other hand, the definition of $B$ is equivalent to
\[
   B(\Psi(v))=\frac{2q+3}{q^{5/2}}v.
\]
Since $R(v)v=qv$, we obtain
\[
   \sqrt q\,R(v)B(\Psi(v))=\frac{2q+3}{q}v,
\]
while the second term in $\widetilde B$ equals $-(2q+3)q^{-1}v$.  Thus
$\widetilde B=0$, proving the lemma.
\end{proof}

\begin{remark}\label{rem:direct}
In fact, the above lemma can also be checked by the chain rule alone: with $D\Phi=\la p^{-1}(\Id-v\otimes v)$, the definitions \eqref{eq:AB-def} are equivalent to the two coefficient identities
\[
   D\Phi\,A\,(D\Phi)^{\mathsf T}=\Id,
   \qquad
   \bigl(D\Phi\,B\bigr)_k+\tr\!\bigl(A\,D_p^2\Phi_k\bigr)=0
   \quad(k=1,2,3),
\]
which hold identically in $p$.
\end{remark}

\section{The counterexample}\label{sec:counterexample}

We now use the exact one-frequency Kolmogorov construction of \cite{DW26}.  For $N\ge1$, define
\begin{equation}\label{eq:IN}
   I_N(\eta)=\int_0^\infty
      \exp\!\left(-\frac{N^2s^3}{3}\right)\e^{-\ii Ns\eta}\,\dd s
\end{equation}
and
\begin{equation}\label{eq:rel-fN}
\begin{aligned}
   u_N(x,v)&=\operatorname{Re}\!\left(\e^{\ii Nx_1}I_N(v_1)\right),\\
   f_N(x,p)&=u_N(x,\Phi(p)),\\
   s_N(x)&=\cos(Nx_1).
\end{aligned}
\end{equation}
The identity
\[
   I_N''(\eta)-\ii N\eta I_N(\eta)=-1,
\]
obtained in the proof of \cite[Theorem~2.1]{DW26} by differentiating the integrand in $s$, shows that $u_N$ solves \eqref{eq:kolmo}, and Lemma~\ref{lem:conjugation} shows that $f_N$ solves \eqref{eq:exact-conj} with the fixed coefficients \eqref{eq:AB-def}.

\begin{theorem}[Relativistic momentum-only counterexample]\label{thm:rel-counterexample}
Let $\alpha\in(0,1)$.  There are constants $c,C>0$ and $N_0\ge1$ such that, for every $N\ge N_0$, the smooth stationary solution $f_N$ in \eqref{eq:rel-fN} satisfies on
\[
   \mathcal Q_1=(-1,0]\times B_1^x\times B_1^p
\]
the bounds
\[
   \norm{f_N}_{L^\infty(\mathcal Q_1)}\le CN^{-2/3},
   \qquad
   [s_N]_{C_p^\alpha(\mathcal Q_1)}=0,
\]
and
\[
   [D_p^2f_N]_{C_p^\alpha(\mathcal Q_{1/2})}
   \ge cN^{\alpha/3}.
\]
Consequently, for the fixed smooth coefficients $A,B$ in \eqref{eq:AB-def}, no estimate of the form
\begin{equation}\label{eq:rel-false-est}
   [D_p^2f]_{C_p^\alpha(\mathcal Q_{1/2})}
   \le C_{A,B,\alpha}
   \left(\norm f_{L^\infty(\mathcal Q_1)}+[s]_{C_p^\alpha(\mathcal Q_1)}\right)
\end{equation}
can hold for all smooth solutions of \eqref{eq:rel-linear}.  More generally, the family rules out any proposed local Schauder estimate whose right-hand side involves only quantities that remain uniformly bounded along it; this includes momentum H\"older norms of the coefficients and forcing of any order, with no spatial modulus.  Corollary~\ref{cor:no-momentum-only-nonlinear-control} below makes this precise.
\end{theorem}

\begin{proof}
We see that the $L^\infty$ estimate follows directly from
\[
   |I_N(\eta)|\le\int_0^\infty\e^{-N^2s^3/3}\,\dd s
   =C_0N^{-2/3},
   \qquad
   C_0:=\int_0^\infty\e^{-r^3/3}\,\dd r=3^{-2/3}\Gamma(1/3),
\]
after the substitution $r=N^{2/3}s$.
The forcing is independent of $p$, so its momentum H\"older seminorm vanishes. Next, we set
\[
   J(y)=\int_0^\infty\e^{-r^3/3}\e^{-\ii ry}\,\dd r.
\]
Then,
\[
   I_N(\eta)=N^{-2/3}J(N^{1/3}\eta),
   \qquad
   I_N''(\eta)=J''(N^{1/3}\eta).
\]
Moreover,
\[
   J''(0)=-\int_0^\infty r^2\e^{-r^3/3}\,\dd r=-1,
\]
and, since $r^2\e^{-r^3/3}\in L^1(0,\infty)$, the Riemann--Lebesgue lemma gives $J''(Y)\to0$ as $Y\to\infty$.  Fix $Y>0$ so that
\[
   \abs{\operatorname{Re}J''(Y)-\operatorname{Re}J''(0)}\ge\frac12.
\]
Let
\[
   v_N=YN^{-1/3}e_1,
   \qquad
   p_N=\Psi(v_N)=\frac{YN^{-1/3}}{\sqrt{1-Y^2N^{-2/3}}}e_1.
\]
Then $|p_N|\asymp N^{-1/3}$ and $p_N\in B_{1/2}$ for large $N$. Along the first coordinate axis, the scalar map
\[
   \phi(p_1)=\frac{p_1}{\sqrt{1+p_1^2}}
\]
satisfies
\[
   \phi'(p_1)=(1+p_1^2)^{-3/2},
   \qquad
   \phi''(p_1)=-3p_1(1+p_1^2)^{-5/2}.
\]
At $x=0$,
\[
   \partial_{p_1p_1}f_N(0,p_1e_1)
   =\operatorname{Re}\!\left(
      I_N''(\phi(p_1))\phi'(p_1)^2
      +I_N'(\phi(p_1))\phi''(p_1)
   \right).
\]
At $p_1=0$, this equals $\operatorname{Re}J''(0)$.  At $p_1=|p_N|$, one has $\phi(p_1)=YN^{-1/3}$,
\[
   \phi'(p_1)^2=1+O(N^{-2/3}),
   \qquad
   \phi''(p_1)=O(N^{-1/3}),
\]
and
\[
   I_N'(YN^{-1/3})=N^{-1/3}J'(Y).
\]
Therefore
\[
   \partial_{p_1p_1}f_N(0,p_N)
   =\operatorname{Re}J''(Y)+O(N^{-2/3}).
\]
For large $N$,
\[
   \abs{\partial_{p_1p_1}f_N(0,p_N)-\partial_{p_1p_1}f_N(0,0)}
   \ge\frac14.
\]
Since $|p_N|\asymp N^{-1/3}$,
\[
   [D_p^2f_N]_{C_p^\alpha(\mathcal Q_{1/2})}
   \ge c|p_N|^{-\alpha}
   \ge cN^{\alpha/3}.
\]
The contradiction to \eqref{eq:rel-false-est} is immediate.
\end{proof}

\begin{remark}\label{rem:consistency}
Theorem~\ref{thm:rel-counterexample} is consistent with the Schauder estimate of \cite[Theorem~6.5]{HSTT}, which retains a Lorentzian spatial modulus.  The forcing $s_N$ has spatial frequency $N$, and its spatial H\"older seminorm at the kinetic exponent,
\[
   [s_N]_{C_x^{\alpha/3}(\mathcal Q_1)}\asymp N^{\alpha/3},
\]
has precisely the size of the Hessian growth.  The construction thus identifies the spatial modulus as the quantity that cannot be deleted from the right-hand side: the result rules out the endpoint obtained by removing all spatial regularity.
\end{remark}

\begin{remark}[Lorentzian H\"older geometry]\label{rem:lorentz}
The oscillation in Theorem~\ref{thm:rel-counterexample} is also visible in the Lorentzian H\"older seminorm of \cite{HSTT}.  For the pairs used in the proof, the two points share $(t,x)$ and one momentum is zero, and the definition of the Lorentzian distance in \cite[(5.6)]{HSTT} gives exactly
\[
   d_L\bigl((t,x,p),(t,x,0)\bigr)=|p|.
\]
Hence the same $N^{\alpha/3}$ lower bound holds with the Lorentzian distance in the denominator.
\end{remark}

\begin{corollary}\label{cor:no-momentum-only-nonlinear-control}
Let $\alpha,\beta\in(0,1)$ and set
\[
   \|h\|_{C_p^\beta(Q)}:=\|h\|_{L^\infty(Q)}+[h]_{C_p^\beta(Q)}.
\]
There is no function
\[
   \mathcal F:[0,\infty)^4\longrightarrow[0,\infty)
\]
that is bounded on bounded subsets and for which
\begin{equation}\label{eq:no-nonlinear-momentum-control}
   [D_p^2f]_{C_p^\alpha(\mathcal Q_{1/2})}
   \le
   \mathcal F\!\left(
      \|f\|_{L^\infty(\mathcal Q_1)},
      \|s\|_{C_p^\beta(\mathcal Q_1)},
      \|A\|_{C_p^\beta(B_1)},
      \|B\|_{C_p^\beta(B_1)}
   \right)
\end{equation}
holds for every smooth solution of \eqref{eq:rel-linear} on $\mathcal Q_1$, even within a fixed uniformly elliptic class.
\end{corollary}

\begin{proof}
Apply \eqref{eq:no-nonlinear-momentum-control} to the fixed coefficients
\eqref{eq:AB-def} and the family from Theorem~\ref{thm:rel-counterexample}.
To keep the solution norm in a fixed compact interval, set
\[
   \widehat f_N:=1+f_N.
\]
The constant function is annihilated by the transport, diffusion, and drift
terms, so $\widehat f_N$ solves the same equation with source $s_N$.  For all
large $N$,
\[
   \frac12\le \widehat f_N\le\frac32.
\]
Furthermore,
\[
   \|s_N\|_{C_p^\beta(\mathcal Q_1)}
   =\|s_N\|_{L^\infty(\mathcal Q_1)}=1,
\]
because $s_N$ is independent of $p$, and the two coefficient norms in
\eqref{eq:no-nonlinear-momentum-control} are fixed finite constants.  Hence all
four arguments of $\mathcal F$ remain in a fixed bounded subset of
$[0,\infty)^4$.  On the other hand,
\[
   [D_p^2\widehat f_N]_{C_p^\alpha(\mathcal Q_{1/2})}
   =[D_p^2f_N]_{C_p^\alpha(\mathcal Q_{1/2})}
   \ge cN^{\alpha/3}\longrightarrow\infty,
\]
contradicting the assumed boundedness of $\mathcal F$ on bounded subsets.
\end{proof}

\begin{proposition}[A zero-source variant]\label{prop:rel-zero-source}
Let $\alpha\in(0,1)$.  For all sufficiently large $N$, there are smooth stationary solutions
$\widetilde f_N$ of
\begin{equation}\label{eq:rel-zero-source}
   \partial_t\widetilde f_N+\frac p{\la p}\cdot\nabla_x\widetilde f_N
   =\tr\!\left(A(p)D_p^2\widetilde f_N\right)
    +B(p)\cdot\nabla_p\widetilde f_N+c_N(x,p)\widetilde f_N
\end{equation}
on $\mathcal Q_1$, with zero external source, such that
\[
   \frac12\le \widetilde f_N\le\frac32,
   \qquad
   \norm{c_N}_{L^\infty(\mathcal Q_1)}\le2,
   \qquad
   [c_N]_{C_p^\alpha(\mathcal Q_1)}
      \le C_\alpha N^{(\alpha-2)/3}\longrightarrow0,
\]
whereas
\[
   [D_p^2\widetilde f_N]_{C_p^\alpha(\mathcal Q_{1/2})}
   \ge cN^{\alpha/3}.
\]
Thus the relativistic momentum-only endpoint also fails when the source is identically zero and the hidden spatial oscillation is placed in a uniformly bounded zeroth-order coefficient whose momentum H\"older seminorm vanishes in the limit.
\end{proposition}

\begin{proof}
Let $f_N$ and $s_N$ be given by \eqref{eq:rel-fN}.  Since
$\norm{f_N}_{L^\infty}\le CN^{-2/3}$, for large $N$ we may set
\[
   \widetilde f_N:=1+f_N,
   \qquad
   c_N:=\frac{s_N}{1+f_N}.
\]
The constant function $1$ is annihilated by the differential part of the operator in
\eqref{eq:rel-zero-source}, so the equation follows from the equation for $f_N$.
The pointwise bounds are immediate. It remains to estimate the momentum seminorm of $c_N$.  Since $s_N$ is independent of $p$,
\[
   |c_N(x,p)-c_N(x,p')|
   \le4|f_N(x,p)-f_N(x,p')|.
\]
The map $\Phi(p)=p/\la p$ is Lipschitz on $B_1$, and the scaling in
\eqref{eq:IN} gives
\[
   [I_N]_{C^\alpha(\R)}
   \le N^{-2/3}N^{\alpha/3}[J]_{C^\alpha(\R)}
   \le C_\alpha N^{(\alpha-2)/3},
\]
where $[J]_{C^\alpha(\R)}$ is finite because $J$ and $J'$ are bounded.
Consequently
\[
   [c_N]_{C_p^\alpha(\mathcal Q_1)}
   \le C_\alpha N^{(\alpha-2)/3}.
\]
Finally, $D_p^2\widetilde f_N=D_p^2f_N$, so the lower bound follows from
Theorem~\ref{thm:rel-counterexample}.
\end{proof}

\end{document}